\documentclass[11pt]{amsart}

\usepackage{amssymb,amsmath,amsfonts,amssymb}
\usepackage{graphics,graphicx,color}
\usepackage{amssymb}
\usepackage{amsmath}
\usepackage{amsthm}
\usepackage{epsfig}
\usepackage{mathrsfs}
\usepackage{graphicx}

\def\RR{\mathbb{R}}
\def\R{\mathbb{R}}
\def\fL{f_L}
\def\fnu{f_\nu}

\def\gn{\gamma_\nu} 

\def\div{\mathrm {div}\,}

\def\nl{\nabla_\lambda}
\def\Ll{L_\lambda}
\def\dH{\dot{H}}

\def\to{\rightarrow}

\def\P{\mathbb{P}}

\def\dt{\partial_t}

\def\na{\nabla}
\def\eps{\varepsilon}

\def\ds{\displaystyle}

\def\tabla{\overline{\na}}

\newtheorem{theo}{Theorem}[section]
\newtheorem{lemma}[theo]{Lemma}
\newtheorem{prop}[theo]{Proposition}
\newtheorem{definition}[theo]{Definition}

\newtheorem{rk&ex}[theo]{Remarks \& Examples}

\title[Global solutions to the  Quasi-Geostrophic equation]{Global weak solutions to the inviscid 3D Quasi-Geostrophic equation}

\author[Puel]{Marjolaine Puel}
\address[Marjolaine Puel]{\newline Laboratoire J.A. Dieudonn\'e  \newline Universit\'e de Nice-Sophia-Antipolis, Nice, France}
\email{Marjolaine.PUEL@unice.fr}

\author[Vasseur]{Alexis F. Vasseur}
\address[Alexis F. Vasseur]{\newline Department of Mathematics, \newline The University of Texas at Austin, Austin, TX 78712, USA}
\email{vasseur@math.utexas.edu}

\date{\today}


\thanks{\textbf{Acknowledgment.} A. F. Vasseur was partially supported by the NSF Grant DMS 1209420. 
}
\begin{document}

\begin{abstract} 
In this article, the authors prove the existence of global weak solutions to the inviscid three-dimensional quasi-geostrophic equation. This equation models the evolution of the temperature on the surface of the earth. It is widely used in geophysics and meteorology. 
\end{abstract}

\maketitle \centerline{\date}

\section{Introduction}

In oceanography, the motion of the atmosphere follows the so-called fundamental equation. This is  the 3D Navier-Stokes equation with the effect of the rotation of the earth (Rosby effect). At large scale, this Rosby effect is very important. Asymptotically, this leads to the so-called geostrophic balance which enforces  the wind velocity to be  orthogonal to the gradient of the pressure in the atmosphere (see Salmon \cite{Salmon}, and Pedlosky \cite{Pedlosky}).  Asymptotic analysis can be performed (see Bourgeois and Beale \cite{Beale}, and Desjardins and Grenier \cite{Grenier}) to derive the quasi-geostrophic equation model (QG), which is not as complex as the fundamental equation, and not as trivial as the geostrophic balance, and still captures the large scale motion of the atmosphere. This model is extensively used in computations of oceanic and  atmospheric circulation, for instance,  to simulate global warming (see Abramov and Majda \cite{Abramov}, and Williams, Read, 
 and Haine \cite{Williams}).  
The quasi-geostrophic equation can be stated as follows. Let $\Psi$ be the stream function for the geostrophic flow. That is, the 3D velocity $(w,U)=(w,u,v)$ has its horizontal  component verifying
$$
(u,v)=(-\partial_{x_2} \Psi,\partial_{x_1}\Psi), \mathrm{\ \ or \ in \ short: \ } U=\tabla^{\perp}\Psi, 
$$
where we denote 
$$
\tabla \Psi=(0,\partial_{x_1} \Psi,\partial_{x_2}\Psi).
$$
Note that, as a convention,   we choose  the  first component, of any vector in the upper half space, to be  the vertical component.
From the model,  the buoyancy is given by
$$
\Theta=\partial_z\Psi.
$$
Let us denote the elliptic operator
$$
L_\lambda\Psi=\partial_z(\lambda\partial_z\Psi)+\bar{\Delta}\Psi,
$$
where $\bar{\Delta}$ stands for the Laplacian in two dimensions:
$$
\bar{\Delta}\Psi=
\partial_{x_1}^2\Psi+\partial_{x_2}^2\Psi,
$$
and  $\lambda=-1/\Theta^0_z$, is a given function, of $z$ only, associated to the buoyancy of a reference state. 
If we denote 
$$\nl\phi =( \lambda \partial_z\phi, \partial_{x_1} \phi, \partial_{x_2} \phi),
$$ 
we have
$$
\Ll \phi=\div (\nl \phi).
$$

Then, the whole dynamic is encoded in the function $\Psi$ which is governed by the following initial boundary problem:
\begin{eqnarray}
 \label{eq:QG}
&&(\dt +\tabla^\perp\Psi\cdot\nabla) (L_\lambda\Psi+\beta_0 x_2)=\fL, \qquad t>0,\ \   z>0,\ \  x\in\RR^2,\\
 \label{eq:QGtrace}
&&(\dt +\tabla^\perp\Psi\cdot\nabla) \gn(\nl \Psi)=\fnu, \qquad  t>0, \ \  z=0, \ \  x\in\RR^2,\\
\label{eq:QGinitial}
&& \Psi(0,z,x)=\Psi^0(z,x). \qquad t=0,\ \  z>0,\ \  x\in\RR^2. 
\end{eqnarray}

The parameter $\beta_0$ comes from the usual $\beta$-plane approximation.  The term $\gn(\nl\Psi)$ stands for the Neumann condition at $z=0$ associated to the operator $\Ll \Psi$.  If $\lambda$ is regular, this coincides with $-\lambda(0) \partial_z \Psi(0,\cdot)$. $\fL$ and $\fnu$ are given exterior forcing.
This equation corresponds to the inviscid version of the quasi-geostrophic equation where the Ekman pumping effect is neglected. 
The Ekman pumping comes from a turbulent viscosity at the surface of the earth due to friction. It adds a viscous terms of the form $r\partial^2_{zz}\Psi$ on the right hand side of Equation (\ref{eq:QGtrace}).

Both, the value of the elliptic operator $\Ll\Psi$, and the Neumann condition $\gn(\nl\Psi)$ at the boundary $z=0$, are advected by the stratified flow with velocity $U=\tabla^\perp\Psi$. At each time, $\Psi$ can be recovered, solving the boundary value elliptic equation.

The main difficulty, while dealing with this equation,  is due to the treatment of the boundary condition.
When the boundary is trivial, that is $\gn (\nl\Psi^0)=0$ at $z=0$, global classical solutions have been constructed in \cite{Beale}. In the case with boundary,
as to now, weak solutions have been constructed only in presence of the regularization effect of  the Ekman pumping  (see \cite{Grenier}). 

We assume that  the function $\lambda$ depends on $z$ only and is globally  bounded by above and by below away from zero. Namely, we assume there exists $\Lambda>0$ such that 
\begin{equation}\label{eq:lambda}
\frac{1}{\Lambda}\leq \lambda(z)\leq\Lambda,\qquad {\mathrm for } \ \ z\in \R^+. 
\end{equation}
\vskip0.3cm
Taking advantage of the incompressibility of the flow, both in $\R^+\times\R^2$ and at $z=0$ in $\R^2$, we say that $\Psi$ is a weak solution to (\ref{eq:QG}) (\ref{eq:QGtrace}) (\ref{eq:QGinitial}) if for every $R>0$ and $T>0$,  $\phi\in C^\infty(\R^4)$ supported in $(-T,T)\times(-R,R)^3$,  and $\bar{\phi}\in C^\infty(\R^3)$ supported in $(-T,T)\times(-R,R)^2$, we have:
\begin{eqnarray*}
&&-\int_0^T\int_0^\infty\int_{\R^2}\left[\left\{\dt \phi+\bar{\nabla}^\perp\Psi\cdot\bar{\nabla}\phi\right\}\left\{\Ll\Psi+\beta x_2\right\}-\fL\phi\right]\,dx\,dz\,dt\\
&&\qquad\qquad =\int_0^\infty\int_{\R^2}\phi(0,z,x)\left(\Ll\Psi^0+\beta x_2\right)\,dx\,dz,
\end{eqnarray*}
\begin{eqnarray*}
&&-\int_0^T\int_{\R^2}\left[\left\{\dt \bar{\phi}+\gamma_0(\bar{\nabla}^\perp\Psi)\cdot\bar{\nabla}\bar{\phi}\right\}\gn(\nl\Psi)-\fnu\bar{\phi}\right]\,dx\,dz\,dt\\
&&\qquad\qquad =\int_{\R^2}\bar{\phi}(0,x)\gn(\nl \Psi^0)\,dx,
\end{eqnarray*}
where $\gamma_0$ is the trace operator at $z=0$ verifying $\gamma_0(\phi)(x)=\phi(0,x)$ for any smooth function $\phi$ defined on $\R^+\times\R^2$. Note that for a vector valued function $\Phi\in[C^\infty(\R^+\times\R^2)]^3$, $\gn(\Phi)$ is the opposite of the vertical component of $\gamma_0(\Phi)$.
\vskip0.3cm
This paper is dedicated to the proof of the following result.
\begin{theo}\label{theo:main}
Assume that for every $T>0$, the source terms verify
\begin{eqnarray*}
&& \fL\in L^1(0,T;L^2(\R^+\times\R^2))\cap  L^1(0,T;L^{6/5}(\R^+\times\R^2)),\\
&& \fnu\in L^1(0,T;L^2(\R^2))\cap  L^{1}(0,T;L^{4/3}(\R^2)).
\end{eqnarray*}
 Consider  an initial value $\Psi^0$ such that
 \begin{eqnarray*}
  &&L_\lambda \Psi^0\  \mathrm{and} \    \nl \Psi^0 \ \mathrm{are \ in \  } L^2(\R^+\times\R^2),\qquad \gn (\nl\Psi^0)\in L^2(\R^2).
  \end{eqnarray*}
  
   Then, there exists $\Psi$  weak solution to (\ref{eq:QG}) (\ref{eq:QGtrace}) (\ref{eq:QGinitial}) on $(0,\infty)\times\R^+\times\R^2$, such that for every $T>0$,    $\nl\Psi \in L^\infty(0,T; L^2(\R^+\times\R^2))\cap C^0(0,T; L^2_{\mathrm{loc}}(\R^+\times\R^2))$,  with :
 \begin{eqnarray*}
 &&\|\Ll\Psi\|_{L^\infty(0,T;L^2(\R^+\times\R^2))}+\|\nabla\Psi\|_{L^\infty(0,T;L^2(\R^+\times\R^2))}\\
 &&\qquad+\|\gn(\nl\Psi)\|_{L^\infty(0,T;L^2(\R^2))}+\|\nabla\Psi\|_{L^\infty(0,T;L^3(\R^+\times\R^2))}\\
 &&\qquad +\|\nabla\Psi\|_{L^\infty(0,T;L^4(\R^+,L^{8/3}(\R^2)))}+\|\nabla\Psi\|_{L^\infty(0,T;L^{\infty}(\R^+;L^2(\R^2)))}\\
 && \leq C \Big(\|\nabla\Psi^0\|_{L^2(\R^+\times\R^2)})+\|\Ll\Psi^0\|_{L^2(\R^+\times\R^2)}+\|\gn(\nl\Psi^0)\|_{L^2(\R^2)}\\
&&\qquad\qquad +\|\fL\|_{ L^1(0,T;L^2(\R^+\times\R^2))}+\|\fL\|_{ L^1(0,T;L^{6/5}(\R^+\times\R^2))}\\
&&\qquad \qquad \left.+\|\fnu\|_{ L^1(0,T;L^2(\R^2))} +\|\fnu\|_{ L^{1}(0,T;L^{4/3}(\R^2))} \right).
 \end{eqnarray*}  
\end{theo}
Note that the estimates, in the theorem, ensure that the weak formulation of  (\ref{eq:QG}) (\ref{eq:QGtrace}) (\ref{eq:QGinitial})  is well defined.
\vskip0.3cm

The main difficulty is that we cannot obtain compactness on the trace $\gn(\nl \Psi)$. However, thanks to a reformulation of the problem (\ref{eq:QG}) (\ref{eq:QGtrace}) (\ref{eq:QGinitial}) (see  Section \ref{section:reformulation}), we can obtain the following stability result.
\begin{theo}\label{prop:stability}
Take  $1<q\leq\infty$, and $T>0$. Consider $\fL^n$ and $\fnu^n$ uniformly bounded, respectively, in $ L^q(0,T;L^2(\R^+\times\R^2))\cap  L^q(0,T;L^{6/5}(\R^+\times\R^2))$ and in $ L^q(0,T;L^2(\R^2))\cap  L^{q}(0,T;L^{4/3}(\R^2))$. 
Let $\Psi_n$ be a sequence of solutions to  (\ref{eq:QG}) (\ref{eq:QGtrace}) (\ref{eq:QGinitial}), with initial value $\Psi^0_n$, and source terms $\fL^n$ and $\fnu^n$, such that 
$\nl \Psi_n\in C^0(0,T;L^2_{\mathrm{loc}}(\R^+\times\R^2))$, $\Ll \Psi_n$ and $\gn(\nl \Psi_n)$ are uniformly bounded, respectively, in the spaces $L^\infty(0,T; L^2(\R^+\times\R^2))$, $L^\infty(0,T; L^2(\R^2\times\R_+))$, and $ L^\infty(0,T;L^2(\R^2))$. Assume that the initial values $\nl\Psi^0_n$ converges to $\nl\Psi^0$ in $L^2(\R^+\times\R^2)$. Then, there exists $\Psi$ solution 
to  (\ref{eq:QG}) (\ref{eq:QGtrace}) (\ref{eq:QGinitial}), 
such that, up to a subsequence, $\nabla \Psi_n$ converges to $\nabla \Psi$ in 
$C^0(0,T; L^2_{\mathrm{loc}}(\R^+\times \R^2))$.
\end{theo}
Note that we do not claim that the trace $\gn(\nl\Psi_n)$ converges strongly in $L^2$ to $\gn(\nl\Psi)$. So we cannot pass directly in the limit in Equation (\ref{eq:QGtrace}). However, thanks to the reformulation problem, it can be ensured that the trace of the limit verifies Equation  (\ref{eq:QGtrace}).
\vskip0.2cm

In the 90's, Constantin, Majda,  and Tabak  \cite{Tabak} began a thorough study of a simplified case. First, fix  $\beta_0=0$, $L=\Delta$, the usual Laplacian (that is $\lambda(z)\equiv1$). Note that if the initial value 
$\Psi^0$ is harmonic, then the first transport equation ensures that it stays harmonic for all times. 
This case is very interesting. While simplifying a lot the equation inside the atmosphere, it keeps all the difficulties due to the boundary. This model, known now as the surface quasi-geostrophic  equation (SQG), can be expressed at the boundary only, as follows. Consider $\theta=\Psi_z$ at $z=0$.
Then $\theta$ is solution to
\begin{eqnarray}
&&\dt\theta+U.\nabla\theta=0,\qquad t>0, (x,y)\in\RR^2,\\
&&\theta=\theta_0,\qquad t=0, (x,y)\in\RR^2,
\end{eqnarray}
and the velocity $U$ can be expressed in $\RR^2$, via a nonlocal operator, as
$$
U=\nabla^\perp \Delta^{-1/2}\theta.  
$$
This model has been popularized as a toy problem for 3D fluid mechanics (see Constantin \cite{Constantin_Lecture} and Held, Pierrehumbert,  Garner, and Swanson, \cite{Held}). The equivalent of Theorem \ref{theo:main} and Theorem  \ref{prop:stability} for the SQG equation is proven in \cite{Tabak}, using different techniques.

\section{Reformulation of the problem}\label{section:reformulation}

For any  $s, k\in \R$, $s\leq 1$, we denote 
\begin{eqnarray*}
&&\dH^s(\R^2)=\{u\in \mathcal{D}'(\R^2); \ \ \ \bar{\Delta}^{s/2} u\in L^2(\R^2)\},\\
&&H^k(\R^2)=\{u\in \mathcal{D}'(\R^2); \ \ \ (\bar{\Delta}-1)^{s/2} u\in L^2(\R^2)\}.
\end{eqnarray*}
\vskip0.3cm

We consider the following Hodge decomposition in $L^2(\R^+\times\R^2)$:
For any $u\in L^2(\R^+\times\R^2)$, there exists a unique $\nl\phi\in L^2(\R^+\times\R^2)$, and a unique $\mbox{curl}v\in L^2(\R^+\times\R^2)$, with $\mbox{curl}v\cdot\nu=0$ at $z=0$, such that 
$$
u=\nabla_\lambda\phi+\mbox{curl}v=\P_\lambda u+\P_{\mathrm{curl}} u.
$$ 
This Hodge decomposition with $\lambda\equiv 1$ is the classical one used to construct the Euler equation in the half space (see Temam \cite{Temam}). It can be extended to general $\lambda$ (see Lions \cite{Lions}). We recall that the trace of the vertical component of  $\mbox{curl}v$ exists in $\dH^{-1/2}(\R^2)$, since $\div (\mbox{curl}v)=0\in L^2(\R^+\times\R^2)$ (see Lemma \ref{lemm:trace}), while the trace of $u$ and $\nl\phi$ cannot be defined in general. The decomposition defines two projection operators well defined on $L^2(\R^+\times\R^2)$. For the sake of completeness, the decomposition is carefully constructed below. We also extend it to any  spaces $L^2(\R^+;\dH^s(\R^2)+H^k(\R^2))$ for $s\leq1$ and $k\in\R$.

\vskip0.3cm
In this section we will show that Problem (\ref{eq:QG}) (\ref{eq:QGtrace}) (\ref{eq:QGinitial}) is equivalent to the following problem:

\begin{equation}\label{eq:QGreformule}
\left\{
\begin{array}{l}
\partial_t\nabla_\lambda\Psi+\P_{\lambda}(\bar\nabla\Psi^\perp\cdot\bar\nabla\nabla_\lambda\Psi)=\beta\  \P_{\lambda}(\Psi e_1)+\nl F,\quad \mbox{on } \RR^+\times \RR^2\times\R^+\\[0.15cm]
\nabla_\lambda\Psi_{|_{t=0}}=\nabla_\lambda\Psi^0\qquad \mbox{ on  } \R^+\times\R^2,
\end{array}\right.\end{equation}
where $e_1=(0,1,0)$ is the first direction in the horizontal plane, and $F(t,\cdot,\cdot)$ is the solution to the elliptic equation with Neumann condition:
\begin{equation}\label{defF}
\Ll F=\fL, \ \ \mathrm{in} \  \ \R^+\times \R^2,\qquad \gn(\nl F)=\fnu, \ \ \mathrm{on} \ \R^2.
\end{equation}
\vskip0.3cm
We say that $\nl\Psi$ is a weak solution to (\ref{eq:QGreformule}), if for any $R>0$, $T>0$, and any $\phi\in C^\infty(\R^4)$ supported in $(-T,T)\times (-R,R)^3:$
\begin{eqnarray*}
&&-\int_0^T\int_0^\infty\int_{\R^2}\nl\Psi\cdot\left\{\dt\nabla\phi+\bar{\nabla}^\perp\Psi\cdot\bar{\nabla}\nabla\phi\right\}\,dx\,dz\,dt\\
&&\qquad -\beta\int_0^T\int_0^\infty\int_{\R^2}\left(\partial_1\phi\ \Psi+\nabla\phi\cdot\nl F\right)\,dx\,dz\,dt\\
&&\qquad\qquad=\int_0^\infty\int_{\R^2}\nabla\phi(0,z,x)\cdot\nl\Psi^0\,dx\,dz.
\end{eqnarray*}
The following theorem shows that the Problem  (\ref{eq:QG}) (\ref{eq:QGtrace}) (\ref{eq:QGinitial}) and  (\ref{eq:QGreformule}) are equivalent and that the weak formulation of (\ref{eq:QGreformule}) is a consistent definition of solution in the sense  of distribution of (\ref{eq:QGreformule}). 

\begin{theo}\label{theo:equivalence}
 Let $\Psi$ be such that,  for every  $T>0$,  $\nabla\Psi\in L^\infty(0,T;L^2(\R^+\times\R^2))$, $\Ll\Psi\in L^\infty(0,T;L^2(\R^+\times\R^2))$, and $\gn(\nl\Psi)\in L^\infty(0,T;L^2(\R^2))$ with $\fL\in L^1(0,T;L^2(\R^+\times\R^2))\cap  L^1(0,T;L^{6/5}(\R^+\times\R^2))$,  $\fnu\in L^1(0,T;L^2(\R^2))\cap  L^{1}(0,T;L^{4/3}(\R^2))$, then we have
 \begin{eqnarray*}
 &&\|\bar{\nabla}^\perp\Psi\otimes\nl\Psi\|_{L^\infty(0,T;L^2(\R^+;\dH^{-1/2}(\R^2)))} \\
&& \qquad\qquad \leq C\left( \|\Ll\Psi\|_{L^\infty(0,T;L^2(\R^+\times\R^2))}+\|\gn(\nl\Psi)\|_{L^\infty(0,T;L^2(\R^2))}\right.\\
&& \qquad\qquad\qquad\left.+\|\nl\Psi\|_{L^\infty(0,T;L^2(\R^+\times\R^2))}\right)^2,\\
  \end{eqnarray*}

  Let $F$ be defined by (\ref{defF}), it satisfies
  
  \begin{eqnarray*}&&\|\nl F\|_{L^1(0,T;L^2(\R^+\times\R^2))}\\
 &&\qquad\qquad\leq C\left(\|\fL\|_{L^1(0,T;L^{6/5}(\R^+\times\R^2))}+\|\fnu\|_{L^{1}(0,T;L^{4/3}(\R^2))}\right),
 \end{eqnarray*}
 and the following statements are equivalent:
 \begin{enumerate}
 \item The function  $\Psi$ is a weak solution to  (\ref{eq:QG}) (\ref{eq:QGtrace}) (\ref{eq:QGinitial}).
 \item The function  $\Psi$ is a weak solution to   (\ref{eq:QGreformule}).
\item For any $R>0$ and $T>0$, and any $\phi_1\in C^{\infty}(\R)$ and $\Phi_2\in (C^{\infty}(\R^3))^3$ supported respectively in $(-T,T)$ and $(-R,R)^3$
 \begin{eqnarray*}
 &&-\int_0^T\int_0^\infty\int_{\R^2}\nl\Psi\cdot\Phi_2(z,x)\dt\phi_1(t)\,dx\,dz\,dt\\
&& \qquad+\int_0^T\int_0^\infty \phi_1(t)   _{H^{3/2}(\R^2)}\langle \Phi_2(z,x);\P_\lambda\bar{\div}(\bar{\nabla}^{\perp}\Psi\otimes \nabla\Psi)\rangle_{H^{-3/2}(\R^2)}\,dz\,dt\\
&& =\beta\int_0^T\int_0^\infty  \phi_1(t) _{\dH^{-1}(\R^2)}\langle \Phi_2(z,x);\P_{\lambda}(\Psi e_1)\rangle_{\dH^1(\R^2)}\,dz\,dt\\
&&\qquad +\int_0^T\int_0^\infty\int_{\R^2}\phi_1(t)\Phi_2(z,x)\cdot\nl F\,dx\,dz\,dt\\
&&\qquad\qquad+\int_0^\infty\int_{\R^2}\phi_1(0)\Phi_2(z,x)\cdot\nl\Psi^0\,dx\,dz.
 \end{eqnarray*}

 Especially, $\dt\Psi\in L^\infty(0,T;L^2(\R^+;H^{-3/2}(\R^2)))$ with
 \begin{eqnarray*}
 &&\qquad\qquad \|\dt\Psi\|_{ L^\infty(0,T;L^2(\R^+;H^{-3/2}(\R^2)))}\\
 &&\leq  \|\fL\|_{L^1(0,T;L^{6/5}(\R^+\times\R^2))}+\|\fnu\|_{L^{1}(0,T;L^{4/3}(\R^2))}\\
 &&\qquad\qquad  +C\left(  \|\Ll\Psi\|^2_{L^\infty(0,T;L^2(\R^+\times\R^2))}+\|\gn(\nl\Psi)\|^2_{L^\infty(0,T;L^2(\R^2))}\right.\\
&&\qquad \qquad\qquad\left.+\|\nl\Psi\|_{L^\infty(0,T;L^2(\R^+\times\R^2))}\right).
 \end{eqnarray*}
  \end{enumerate} 
 \end{theo}
Let us give the idea why solutions of (\ref{eq:QGreformule}) verify  (\ref{eq:QG}) (\ref{eq:QGtrace}) (\ref{eq:QGinitial}).
Thanks to the boundedness of $\Ll\Psi$, we can define $\gn(\nl\Psi)$ at $z=0$. In the Hodge decomposition, the vertical component 
of  $\P_{\mbox{curl}}$  is equal to 0 at $z=0$. Hence, formally, the vertical component to  $\P_{\lambda}(\bar\nabla\Psi^\perp\cdot\bar\nabla\nabla_\lambda\Psi)$ at $z=0$ is the same as the vertical component of $\bar\nabla\Psi^\perp\cdot\bar\nabla\nabla_\lambda\Psi$. This provides formally Equation (\ref{eq:QGtrace}).  Similarly, $\div(\P_\lambda \cdot )=\div(\cdot)$. Therefore,  because $\partial_i(\bar\nabla\Psi)^\perp\cdot\bar\nabla\partial_i\Psi=0$ for all $i$, taking the divergence of equation (\ref{eq:QGreformule}) gives (\ref{eq:QG}). 
\vskip0.3cm
Note that for a function  $u\in L^2(\R^+\times\R^2)$, in general, we cannot define the trace $u\cdot\nu$ at $z=0$, and so the trace of 
$\P_\lambda u\cdot\nu$ at $z=0$. However, the projection $\P_\lambda$ is continuous in $L^2(\R^2\times \R^+)$. Noticing that 
$\P_\lambda$ commute with $\bar{\nabla}\cdot$ (but not with $\partial_z$), it is enough to have compactness  for  $\bar\nabla\Psi^\perp\otimes \nabla_\lambda\Psi$ to have stability for Equation (\ref{eq:QGreformule}).
This provides stability for solutions of the quasi-geostrophic equation, even without compactness on the trace. 
\vskip0.3cm
It is interesting to compare Equation (\ref{eq:QGreformule}), in the case $\beta=0$ and  $F=0$,  with the Euler equation in the half space with zero flux at $z=0$:
$$
\begin{array}{l}
\dt \mbox{curl} v+\P_{\mbox{curl}}[\div (\mbox{curl} v\otimes \mbox{curl} v)]=0, \qquad (t,x,z)\in \R^+\times\R^2\times\R^+,\\[0.3cm]
\mbox{curl} v_{t=0}=\mbox{curl} v^0,\qquad (x,z)\in \R^+\times\R^2,
\end{array}
$$
where the velocity is given by $u=\mbox{curl} v$. The quasi-geostrophic equation is obtained by flipping  $\nabla_\lambda$ and $\mbox{curl}$, and by stratifying the flow (the advection is $\bar\nabla^\perp\Psi$).  Note that Equation (\ref{eq:QG}) is reminiscent to the vorticity equation for 2D Euler: one is obtained through the $\mbox{curl}$ operator, while the second one is obtained through the $\div$ operator from (\ref{eq:QGreformule}). 
\vskip0.3cm
Similarly to the introduction of the gradient of pressure for the  Euler equation, Equation (\ref{eq:QGreformule}) can be written as
\begin{equation*}\left\{\begin{array}{l}
\partial_t\nabla_\lambda\Psi+\bar\nabla\Psi^\perp\cdot\bar\nabla\nabla_\lambda\Psi=\rm{curl} Q,\quad \mbox{on }\RR^+\times \RR^+\times \RR^2\\[0.15cm]
\nabla_\lambda\Psi_{|_{t=0}}=\nabla_\lambda\Psi_0\mbox{ given }\\[0.15cm]
\mbox{curl}Q_{|_{z=0}}\cdot\nu=0.
\end{array}\right.\end{equation*}
This formulation will not be used in this paper.

\section{Properties of the Hodge decomposition}

\subsection{Preliminaries}
 We recall that for $s>0$,   $H^{s}(\R^+\times\R^2)$ is the space of  restrictions to $\R^+\times\R^2$ of functions of $H^{s}(\R^3)$.
 Moreover, the canonical norm on $H^{s}(\R^+\times\R^2)$ is equivalent to 
 $$
|||u|||_{s}= \inf\{\|v\|_{H^{s}(\R^3)}; \ \ v\in H^{s}(\R^3)  \ \mathrm{extension \ of \ }u  \ \mathrm{in \ } \R^3\}.
 $$
 We have the same property for $\dH^s(\R^+\times\R^2)$, for $s>0$. 

Hence, from interpolation  there exists a constant $C>0$ such that 
\begin{equation}\label{eq:Interpolation}
\|u\|^2_{H^{1/2}(\R^+\times\R^2)}\leq C(\|u\|^2_{L^2(\R^+; H^{1/2}(\R^2))}+\|u\|_{L^2(\R^+\times\R^2))}\|\partial_zu\|_{L^2(\R^+\times\R^2)}).
\end{equation}  
Moreover, extending $u$ on $\R^3$ by $u(z,x)=u(-z,x)$ for $z<0$, and using Poincar\'e inequality in $\R^3$, we find that there exists $C>0$ such that
\begin{equation}\label{eq:Sobolev}
\|u\|_{L^6(\R^+\times\R^2)}\leq C\|\nabla u\|_{L^2(\R^+\times\R^2)}, \qquad \mathrm{whenever \ } u\in H^1(\R^+\times\R^2).
\end{equation}  

\vskip0.3cm
Following Temam \cite{Temam}, we have
\begin{eqnarray*}
&&H=\{u\in L^2(\R^+\times\R^2); \ \div u=0, \ \gn(u)=0\}\\
&&=\left\{ u\in L^2(\R^+\times\R^2); \  \int_0^\infty\!\!\!\int_{\R^2}u\cdot\nabla\phi\,dx\,dz=0; \ \forall \phi\in \dH^1(\R^+\times\R^2)\right\}. 
\end{eqnarray*}
We denote $[\dH^1(\R^+\times\R^2)]^*$ the dual of $\dH^1(\R^+\times\R^2)$. We recall the following trace properties.
 \begin{lemma}\label{lemm:trace}
  \begin{equation}\label{trace1}
  \|\gamma_0 u\|_{\dH^{1/2}(\R^2)} \leq \|\nabla u\|_{L^2(\R^+\times\R^2)},  \qquad \mathrm{for \ } u\in \dH^1(\R^+\times\R^2),
  \end{equation}
  where $\gamma_0 u(x)=u(0,x)$ for $x\in \R^2$, whenever $v$ is smooth.
   \begin{equation}\label{trace2}
   ||\gn(u)||_{\dot {H}^{-1/2}(\R^2)}\leq ||u||_{L^2(\R^+\times \R^2)}+\|\div\ u\|_{[\dH^1(\R^+\times\R^2)]^*},
   \end{equation}
   for any $u\in[ L^2(\R^2)]^3$ with $\div\ u\in [\dH^1(\R^+\times\R^2)]^*$.
  \end{lemma}
  \begin{proof}
  - Proof of (\ref{trace1}). 
  Let $u$ be a smooth function. From the Dirichlet minimization problem, 
  $\tilde u $ defined by 
  $$
 \left\{ \begin{array}{l}
  -\Delta \tilde u =0\\
\gamma_0(\tilde u)=\gamma_0 (u),
  \end{array}\right.
  $$ 
  minimizes the the Dirichlet integral among functions with same trace at $z=0$.
Hence
  \begin{eqnarray*}
  &&\int_0^\infty\int_{\R^2}  |\nabla u|^2\,dx\,dz\geq \int _0^\infty\int_{\R^2}|\nabla \tilde u|^2\,dx\,dz\\
&&  \qquad =\int_{\R^2}\tilde u(0,x) \gamma_\nu (\nabla \tilde u)\,dx=||\bar{\Delta}^{\frac{1}{2}} \gamma_0(\tilde u)||^2_{L^2(\R^2)}=||\gamma_0(u)||^2_{\dH^{\frac{1}{2}}(\R^2)}.
  \end{eqnarray*}
  The result can be extended to $u\in \dH(\R^+\times\R^2)$ by density.
  
  - Proof of (\ref{trace2}). Let $u,v$ smooth functions, such that $\nabla v\in L^2(\R^+\times\R^2)$, $u\in L^2(\R^+\times\R^2)$, and $\div  u\in [\dH^1(\R^+\times\R^2)]^*$. From the divergence theorem, we have 
  \begin{eqnarray*}
  &&\qquad \int_0^\infty\int_{\R^2}\nabla v\cdot u\,dx\,dz\\
  &&=- \int_0^\infty\int_{\R^2}v(\div  u)\,dx\,dz+\int_{\R^2}v(0,x)\gn(u)\,dx.
  \end{eqnarray*}
  Hence
  \begin{eqnarray*}
  &&\left|\int_{\R^2}\gamma_0(v)\gn(u)\,dx\right|\leq \|\nabla v\|_{L^2(\R^+\times\R^2)}\\
  &&\qquad\qquad\qquad \left\{\|u\|_{L^2(\R^+\times\R^2)}+\|\div\ u\|_{[\dH^1(\R^+\times\R^2)]^*}\right\}.
  \end{eqnarray*}
  But from (\ref{trace1}):
  \begin{eqnarray*}
  &&\|\gn(u)\|_{\dH^{-1/2}(\R^2)}=\sup\left\{\left|\int_{\R^2}v(x)\gn(u)\,dx\right|; \ \|v\|_{H^{1/2}(\R^2)}\leq 1\right\}\\
  &&\qquad \qquad =\sup\left\{\left|\int_{\R^2}\gamma_0(v)\gn(u)\,dx\right|; \ \|\nabla v\|_{L^2(\R^+\times\R^2)}\leq 1\right\}.
  \end{eqnarray*}
  So 
  $$
  \|\gn(u)\|_{\dH^{-1/2}(\R^2)}\leq \|u\|_{L^2(\R^+\times\R^2)}+\|\div\ u\|_{[\dH^1(\R^+\times\R^2)]^*}.
  $$
  We conclude again by density.
   \end{proof}
  
\subsection{The Hodge decomposition}

We show the following proposition.
\begin{prop}\label{prop:HodgeL2}
For any $u\in L^2(\R^+\times\R^2)$, there exists a unique $\nl \psi\in L^2(\R^+\times\R^2)$ and $\mathrm{curl} \ v\in H$ with
$$
u=\nl\psi +\mathrm{curl} \ v.
$$
The linear operators $\P_\lambda u=\nl\psi$ and $\P_{\mathrm{curl}}u=\mathrm{curl} \ v$ are bounded in $L^2(\R^+\times\R^2)$. Especially, there exists $C>0$ such that 
$$
\|\P_{\lambda}u\|_{L^2(\R^+\times\R^2)}\leq C \|u\|_{L^2(\R^+\times\R^2)},
$$
and for every $u\in L^2(\R^+\times\R^2)$,  and any $\phi\in \dH^1(\R^+\times\R^2)$,
$$
\int_0^\infty\int_{\R^2}\nabla \phi(z,x)u(z,x)\,dx\,dz=\int_0^\infty\int_{\R^2}\nabla \phi(z,x)\P_\lambda u(z,x)\,dx\,dz.
$$
\end{prop}
\begin{proof}
We construct $\psi$ as $\psi_1+\psi_2$, where $\psi_1\in \dH^1(\R^+\times\R^2)$ is the solution to the elliptic equation with Dirichlet boundary condition:
$$
\Ll \psi_1=\div \ u,\  \mathrm{in } \ \R^+\times \R^2,\qquad \gamma_0(\psi_1)=0,
$$
and $\psi_2\in \dH^1(\R^+\times\R^2)$ is solution to the elliptic equation with Neumann boundary condition:
$$
\Ll \psi_2=0,\  \mathrm{in } \ \R^+\times \R^2,\qquad \gn(\nl \psi_2)=\gn(\tilde{u}),
$$

where $\tilde{u}=u-\nl \psi_1$. Note that $\div \ \tilde{u}=0$, so from Lemma (\ref{trace2}), $\gn(\tilde u)$ is well defined in $H^{-1/2}(\R^2)$.
\vskip0.3cm
Namely, for every $\phi\in \dH^1(\R^+\times\R^2)$ with $\gamma_0(\phi)=0$
\begin{eqnarray*}
&& \int_0^\infty\int_{\R^2} \nl \psi_1\nabla\phi\,dx\,dz= \int_0^\infty\int_{\R^2}u\nabla\phi\,dx\,dz.
\end{eqnarray*}
And for every $\phi\in \dH^1(\R^+\times\R^2)$
\begin{eqnarray*}
&& \int_0^\infty\int_{\R^2} \nl \psi_2\nabla\phi\,dx\,dz= -\left[   _{H^{-1/2}(\R^2)}\langle \gn(\tilde{u}); \gamma_0(\phi)\rangle_{H^{1/2}(\R^2)}\right].
\end{eqnarray*}
The existence and  uniqueness of $\psi_1$ and $\psi_2$ are obtained thanks to the Lax-Milgram theorem. 
Taking $\phi=\psi_1$, we get 
$$
\|\nabla_{\sqrt{\lambda}}\psi_1\|^2_{L^2(\R^+\times\R^2)}\leq \|\nabla\psi_1\|_{L^2(\R^+\times\R^2)}\|u\|_{L^2(\R^+\times\R^2)}.
$$
Since $(1/\Lambda)|\nabla\psi_1|\leq |\nabla_{\sqrt{\lambda}}\psi_1|\leq \Lambda|\nabla\psi_1|$, we have
$$
\|\nabla\psi_1\|_{L^2(\R^+\times\R^2)}\leq \Lambda \|u\|_{L^2(\R^+\times\R^2)}.
$$
Taking $\phi=\psi_2$, and using the trace theorems,
\begin{eqnarray*}
&&\|\nabla_{\sqrt{\lambda}}\psi_2\|^2_{L^2(\R^+\times\R^2)}\leq \|\gn(\tilde{u})\|_{\dH^{-1/2}(\R^2)}\|\gamma_0(\psi_2)\|_{\dH^{1/2}(\R^2)}\\
&&\qquad\leq  \|\tilde{u}\|_{L^2(\R^+\times\R^2)}\|\nabla\psi_2\|_{L^2(\R^+\times\R^2)}.
\end{eqnarray*}
Hence
\begin{eqnarray*}
&&\|\nabla\psi_2\|_{L^2(\R^+\times\R^2)}\leq \Lambda \|\tilde{u}\|_{L^2(\R^+\times\R^2)}\\
&&\qquad \leq \Lambda (\|u\|_{L^2(\R^+\times\R^2)}+(\|\nabla\psi_1\|_{L^2(\R^+\times\R^2)})\\
&&\qquad \leq \Lambda(1+\Lambda) \|u\|_{L^2(\R^+\times\R^2)}.
\end{eqnarray*}
 From the construction, $u-\nl\psi$ is divergence free, so it can be written as $\mathrm{curl}\ v$. Moreover,  by construction 
  $\gn(\mathrm{curl}\ v)=0$, so $\mathrm{curl}\ v\in H$, and
  \begin{eqnarray*}
&&  \|\mathrm{curl}\ v\|_{L^2(\R^+\times\R^2)}\leq \|u\|_{L^2(\R^+\times\R^2)} +\|\nl\psi\|_{L^2(\R^+\times\R^2)}\\
&&\qquad \qquad \leq C \|u\|_{L^2(\R^+\times\R^2)} .
  \end{eqnarray*}
If there exists $\psi\in \dH^1(\R^+\times\R^2)$ and $\mathrm{curl} \  v\in H$, such that 
$$
\nl\psi+\mathrm{curl} \  v=0,
$$
then, since $\int \mathrm{curl} \  v \cdot\nabla \psi=0$,
$$
\int_0^\infty\int_{\R^2} |\nabla_{\sqrt{\lambda}}\psi|^2\,dx\,dz=0.
$$
So $\nl\psi=\mathrm{curl} \  v=0$, and the decomposition is unique. Hence $\P_\lambda$ and $\P_{\mathrm{curl} }$ are well defined as bounded operator in $L^2(\R^+\times\R^2)$.
\vskip0.3cm
For any $\nabla\phi\in L^2(\R^+\times\R^2)$
$$
\int_{0}^\infty\int_{\R^2} u \cdot\nabla\phi\,dx\,dz=\int_{0}^\infty\int_{\R^2} \P_{\mathrm{curl}}u\cdot\nabla\phi\,dx\,dz+\int_{0}^\infty\int_{\R^2} \P_{\lambda}u\cdot\nabla\phi\,dx\,dz.
$$
But $\P_{\mathrm{curl}}u\in H$, so 
$$
\int_{0}^\infty\int_{\R^2} \P_{\mathrm{curl}}u\cdot\nabla\phi\,dx\,dz=0.
$$
\end{proof}
We denote $\mathcal{S}(\R^2)$ the Schwartz class of $C^\infty$ functions with fast decay at infinity, and $\mathcal{S}'(\R^2)$ the dual space as set of  distributions.  
We now extend the operator $\P_\lambda$ to $L^2(\R^+;H^{m}(\R^2)+\dH^{s}(\R^2))$  into itself, for $m\in \R$, and $s\in \R$ with $s\leq1$.
\begin{prop}\label{prop:Hodgeextended}
For every $m\in\R$, and $s< 1$, we can extend the operator  $\P_\lambda$ from  $L^2(\R^+;H^{m}(\R^2)+\dH^{s}(\R^2))$ into itself, and 
$$
\|\P_\lambda u\|_{L^2(\R^+;H^{m}(\R^2)+\dH^{s}(\R^2))}\leq C \|u\|_{L^2(\R^+;H^{m}(\R^2)+\dH^{s}(\R^2))}.
$$
Moreover, if $u\in \mathcal{S}'(\R^+\times \R^2)$, $u\in L^2(\R^+;H^{-m}(\R^2)+\dH^{1}(\R^2))$ for $m>0$, and  $\phi\in \mathcal{S}(\R^+\times\R^2)$,
$$
\int_0^\infty     \mbox{}_ {\mathcal{S}'(\R^2)}\langle u(z);\nabla\phi(z) \rangle_{\mathcal{S}(\R^2)}\,dz=0,
$$
then 
$$
\P_\lambda u=0.
$$
\end{prop}
\begin{proof}
Note that for a smooth function $u\in \mathcal{S}(\R^+\times\R^2)$, 
$$
\P_\lambda\bar{\Delta}^{m/2}=\bar{\Delta}^{m/2}\P_\lambda,\qquad \P_\lambda(\bar{\Delta}-1)^{m/2}=(\bar{\Delta}-1)^{m/2}\P_\lambda.
$$
So for $u=u_1+u_2$ with $u_1\in L^2(\R^+;H^{m}(\R^2))$ and $u_2\in L^2(\R^+;\dH^{s}(\R^2))$, we define $\P_\lambda u$ as
$$
\P_\lambda u=(\bar{\Delta}-1)^{-m/2}\P_\lambda(\bar{\Delta}-1)^{m/2}u_1+\bar{\Delta}^{-s/2}\P_\lambda\bar{\Delta}^{s/2}u_2.
$$
Now, if $u\in L^2(\R^+;H^{-m}(\R^2)+\dH^{1}(\R^2))$ for $m>0$, then 

$$
u=\bar{\Delta}^{-1/2}(1-\bar{\Delta}^{\frac{m+1}{2}}) g,
$$
with $g\in L^2(\R^+\times\R^2)$. Then from Proposition \ref{prop:HodgeL2}, There exists $\Psi\in \dH^1(\R^+\times\R^2)$ such that
$$
\P_{\lambda}g=\nl\Psi\in L^2(\R^+\times\R^2).
$$
Consider a sequence $\Psi_n\in \mathcal{S}(\R^+\times\R^2)$ such that $\nabla\Psi_n$ converges to $\nabla \Psi$ in $L^2(\R^+\times\R^2)$ when $n$ converges to infinity, and take 
$$
\Phi_n=\bar{\Delta}^{1/2}(1-\bar{\Delta}^{-\frac{m+1}{2}})\Psi_n\in \mathcal{S}(\R^+\times\R^2).
$$
We have 
\begin{eqnarray*}
&&0=\int_0^\infty     \mbox{}_ {\mathcal{S}'(\R^2)}\langle u(z);\nabla\Phi_n(z) \rangle_{\mathcal{S}(\R^2)}\,dz\\
&&\qquad =\int_0^\infty\int_{\R^2}\nabla\Psi_n\cdot g\,dx\,dz\\
&&\qquad =\int_0^\infty\int_{\R^2}\nabla\Psi_n\cdot \nl\Psi\,dx\,dz,
\end{eqnarray*}
thanks to Proposition \ref{prop:HodgeL2}. this term converges to 
$$
\int_0^\infty\int_{\R^2}\nabla\Psi\cdot \nl\Psi\,dx\,dz=\int_0^\infty\int_{\R^2}|\nabla_{\sqrt{\lambda}}\Psi|^2\,dx\,dz.
$$
This term is then equal to zero, and $\nabla_{\sqrt{\lambda}}\Psi=\nabla\Psi=0$.
So 
$$
\P_\lambda u=\bar{\Delta}^{-1/2}(1-\bar{\Delta}^{\frac{m+1}{2}})\P_{\lambda}g=0.
$$
\end{proof}

\section{Proof of  the equivalence  Theorem \ref{theo:equivalence}}

We begin with the following proposition.
\begin{prop}\label{prop:apriori}
Consider $\nl\Psi\in L^2(\R^+\times\R^2)$ with $\Ll\Psi \in L^2(\R^+\times\R^2)$, and $\gn(\nl\Psi)\in L^2(\R^2)$.
Then

1-
\begin{eqnarray*}
&&\|\nl \Psi\|_{L^2(\R^+;H^{1/2}(\R^2))}+\|\gamma_0(\nl \Psi)\|_{L^2(\R^2)}\\
&&\qquad\qquad+\|\nl \Psi\|_{C^0(\R^+;L^2(\R^2))}\\
&& \leq C\left(\|\gn(\nl\Psi)\|_{L^2(\R^2)}+\|\Ll\Psi\|_{L^2(\R^+\times\R^2)}+\|\nl\Psi\|_{L^2(\R^+\times\R^2)}\right).
\end{eqnarray*}

2- By Sobolev embedding, we get
\begin{eqnarray*}
&&\|\nl \Psi\|_{L^3(\R^+\times \R^2)}+\|\nl \Psi\|_{L^4(\R^+;L^{8/3}(\R^2))}\\
&& \leq C\left(\|\gn(\nl\Psi)\|_{L^2(\R^2)}+\|\Ll\Psi\|_{L^2(\R^+\times\R^2)}+\|\nl\Psi\|_{L^2(\R^+\times\R^2)}\right).
\end{eqnarray*}

\end{prop}

\begin{proof}
For every $\phi\in \dH^1(\R^+\times\R^2)$ we have
\begin{eqnarray*}
&&\left|\int_0^\infty\int_{\R^2}\phi \bar{\Delta}^{\frac{1}{2}}(\Ll\Psi)\,dx\,dz\right|\leq \|\bar{\Delta}^{\frac{1}{2}}\phi\|_{L^2}\| \Ll\Psi\|_{L^2}\\
&&\qquad\qquad\leq \|\bar\nabla\phi\|_{L^2}\| \Ll\Psi\|_{L^2}\\
&&\qquad\qquad\leq \|\nabla\phi\|_{L^2}\| \Ll\Psi\|_{L^2}.
\end{eqnarray*}
So $\bar{\Delta}^{\frac{1}{2}}(\Ll\Psi)\in \left( \dH^1(\R^+\times\R^2)\right)^*$. Moreover 
$$
\gn(\nl [\bar{\Delta}^{\frac{1}{4}}\Psi])=\bar{\Delta}^{\frac{1}{4}}\gn(\nl \Psi)\in \dH^{-1/2}(\R^2).
$$
consider $\psi_1\in \dH^1(\R^+\times\R^2)$ solution to 
$$
\Ll\psi_1=\bar{\Delta}^{\frac{1}{2}}(\Ll\Psi), \qquad \gn(\nl\psi_1)=0,
$$
since $\bar{\Delta}^{1/2}$ commutes with $\Ll$ and $\gn$, we have that 
$$
\nl\Psi_1=\bar{\Delta}^{-1/2}\nl\psi_1\in L^2(\R^+;\dH^1(\R^2)),
$$
and 
\begin{eqnarray*}
&&\|\nl\Psi_1\|_{ L^2(\R^+,\dH^1(\R^2))}\leq \sqrt{\Lambda}\|\nabla_{\sqrt{\lambda}}\Psi_1\|_{L^2(\R^+;\dH^{1}(\R^2))}\\
&&\qquad\qquad \leq  \sqrt{\Lambda}\| \Ll\Psi\|_{L^2}.
\end{eqnarray*}
Setting $\Psi_2=\Psi-\Psi_1$, it verifies
$$
\Ll\Psi_2=0,\qquad \gn(\nl\Psi_2)=\gn(\nl\Psi).
$$
The function $\bar{\Delta}^{\frac{1}{2}}\gn(\nl \Psi)\in \dH^{-1/2}(\R^2)$, and $\bar{\Delta}^{\frac{1}{2}}$ commutes with  $\Ll$ and $\gn$,
so $\Psi_2=\bar{\Delta}^{\frac{-1}{2}}\psi_2$ solution to 
$$
\Ll\psi_2=0,\qquad \gn(\nl\psi_2)=\bar{\Delta}^{\frac{1}{2}}\gn(\nl\Psi),
$$
verifying
\begin{eqnarray*}
&&\qquad\qquad \|\nabla\Psi_2\|_{ L^2(\R^+;\dH^{1/2}(\R^2))}\\
&&\leq \sqrt{\Lambda}\|\nabla_{\sqrt{\lambda}}\Psi_2\|_{L^2(\R^+;\dH^{1/2}(\R^2))}\leq \sqrt{\Lambda} \| \gn(\nl\Psi)\|_{L^2(\R^2)}.
\end{eqnarray*}
But, for almost every $z>0$, 
\begin{eqnarray*}
&&\qquad\qquad \|\nl\Psi(z,\cdot)\|^2_{H^{1/2}(\R^2)}\\
&&=\int_{\R^2}|\widehat{\nl\Psi}(z,\xi)|^2(1+|\xi|^2)^{1/2}\,d\xi\\
&&=\int_{|\xi|\leq1}|\widehat{\nl\Psi}(z,\xi)|^2(1+|\xi|^2)^{1/2}\,d\xi\\
&&\qquad+\int_{|\xi|\geq1}|\widehat{\nl\Psi_1}(z,\xi)+\widehat{\nl\Psi_2}(z,\xi)|^2(1+|\xi|^2)^{1/2}\,d\xi\\
&&\leq \sqrt{2}\|\nl\Psi(z)\|^2_{L^2(\R^2)}+2 \sqrt{2}\int_{\R^2}\left(|\xi||\widehat{\nl\Psi_1}(z,\xi)|^2+|\xi|^2|\widehat{\nl\Psi_2}(z,\xi)|^2\right)\,d\xi\\
&&\leq C\left(\|\nl\Psi(z)\|^2_{L^2(\R^2)}+\|\nl\Psi_2(z)\|^2_{\dH^{1}(\R^2)}+\|\nl\Psi_1(z)\|^2_{\dH^{1/2}(\R^2)}\right).
\end{eqnarray*}

Hence
\begin{eqnarray*}
&&\qquad\qquad\|\nl\Psi\|_{L^2(\R^+;H^{1/2}(\R^2))}\\
&&\leq C\left(\|\nl\Psi\|_{L^2(\R^+\times\R^2)}+\|\Ll\Psi\|_{L^2(\R^+\times\R^2)}+ \| \gn(\nl\Psi)\|_{L^2(\R^2)}\right).
\end{eqnarray*}
From Lemma \ref{lemm:trace}, both at $z=0$ and at $z$ for almost every $z>0$, we get
\begin{eqnarray*}
&\|\lambda(z)\partial_z\Psi(z)\|_{L^2(\R^2)}
&=\|\bar{\Delta}^{1/4}(\lambda(z)\partial_z\Psi(z))\|_{\dH^{-1/2}(\R^2)}\\
&&\leq \|\nl\bar{\Delta}^{1/4}\Psi\|_{L^2(\R^+\times\R^2)}+ \|\Ll\bar{\Delta}^{1/4}\Psi\|_{[\dH^1(\R^+\times\R^2)]^*}\\
&&\leq C\left(\|\nl\Psi\|_{L^2(\R^+\times\R^2)}+\|\Ll\Psi\|_{L^2(\R^+\times\R^2)}+ \| \gn(\nl\Psi)\|_{L^2(\R^2)}\right).
\end{eqnarray*}
\vskip0.3cm
Using Lemma \ref{lemm:trace} at $z=0$, and at $z$ for almost every $z>0$, we get
\begin{eqnarray*}
&&\qquad\|\gamma_0(\bar{\nabla}\Psi)\|_{L^2(\R^2)}+\|\bar{\nabla}\Psi\|_{L^\infty(\R^+;L^2(\R^2))}\\
&&\leq\|\gamma_0(\bar{\Delta}^{1/2}\Psi)\|_{L^2(\R^2)}+\|\bar{\Delta}^{1/2}\Psi\|_{L^\infty(\R^+;L^2(\R^2))}\\
&&\leq\|\gamma_0(\bar{\Delta}^{1/4}\Psi)\|_{\dH^{1/2}(\R^2)}+\|\bar{\Delta}^{1/4}\Psi\|_{L^\infty(\R^+;\dH^{1/2}(\R^2))}\\
&&\leq 2\|\nabla \bar{\Delta}^{1/4}\Psi\|_{L^2(\R^+\times\R^2)}\leq 2 \|\nl\Psi\|_{L^2(\R^+;H^{1/2}(\R^2))}\\
&&\leq C\left(\|\nl\Psi\|_{L^2(\R^+\times\R^2)}+\|\Ll\Psi\|_{L^2(\R^+\times\R^2)}+ \| \gn(\nl\Psi)\|_{L^2(\R^2)}\right).
\end{eqnarray*}
\vskip0.3cm
Regularizing $\nl\Psi$ in $x$ and $z$ by convolution, we approximate it by a sequence of function in $C^0(\R^+;L^2(\R^2))$, with convergence in $L^\infty(\R^+;L^2(\R^2))$. So $\nl\Psi \in C^0(\R^+;L^2(\R^2))$.
\vskip0.3cm
From Sobolev imbedding, we have the same control on $\|\nl\Psi\|_{L^2(\R^+;L^4(\R^2))}$. Interpolation between $L^\infty(L^2)$ and 
$L^2(L^4)$ gives the control on the $L^4(L^{8/3})$ and $L^3(L^3)$ norms of $\nl\Psi$.
\end{proof}

We now prove Theorem \ref{theo:equivalence}:
\begin{proof}
Since $L^{4/3}(\R^2)$ is continuously imbedded in $\dH^{-1/2}(\R^2)$, and using  
 Proposition \ref{prop:apriori}, we get 
\begin{eqnarray*}
&&\qquad
\|\bar{\nabla}^\perp\Psi\otimes\nl\Psi\|_{L^\infty(0,T;L^2(\R^+;\dH^{-1/2}(\R^2)))}\\
&&\leq C \|\bar{\nabla}^\perp\Psi\otimes\nl\Psi\|_{L^\infty(0,T;L^2(\R^+;L^{4/3}(\R^2)))}\\
&&\leq  C\|\nl\Psi\|^2_{L^\infty(0,T;L^4(\R^+;L^{8/3}(\R^2)))}\\
&&\leq  C \left(\|\Ll\Psi\|_{L^\infty(0,T;L^2(\R^+\times\R^2))}+\|\gn(\nl\Psi)\|_{L^\infty(0,T;L^2(\R^2))}\right.\\
&&\qquad\qquad\left.+\|\nl\Psi\|_{L^\infty(0,T;L^2(\R^+\times\R^2))}\right)^2.
\end{eqnarray*}

Thanks to (\ref{eq:Sobolev}) $\dH^1(\R^+\times\R^2)$ is continuously embedded in $L^6(\R^+\times\R^2)$, so for every $u\in \dH^1(\R^+\times\R^2)$, and almost every $t>0$
\begin{eqnarray*}
&&\left|\int_0^\infty\int_{\R^2}f(t) u\,dx\,dz\right|\leq \|f(t)\|_{L^{6/5}(\R^+\times\R^2)}\|u\|_{L^6(\R^+\times\R^2)}\\
&&\qquad \leq C \|f(t)\|_{L^{6/5}(\R^+\times\R^2)}\|\nabla u\|_{L^2(\R^+\times\R^2)}.
\end{eqnarray*}
Hence 
\begin{eqnarray*}
&&\|\fL\|_{L^1(0,T;[\dH^1(\R^+\times\R^2)]^*)}\leq C  \|\fL\|_{L^1(0,T;L^{6/5}(\R^+\times\R^2))},\\
&& \|\fnu\|_{L^{1}(0,T;\dH^{-1/2}(\R^2))}\leq C \|\fnu\|_{L^{1}(0,T;L^{4/3}(\R^2))}.
\end{eqnarray*}
Hence, from the Lax Milgram Theorem, there exists a unique $F$ solution to $\Ll F=\fL$, $\gn(\nl F)=\fnu$, with
$$
 \|\nl F\|_{L^1(0,T;L^2(\R^+\times\R^2))}\leq C\left( \|\fL\|_{L^1(0,T;L^{6/5}(\R^+\times\R^2))}+\|\fnu\|_{L^{1}(0,T;L^{4/3}(\R^2))}\right).
$$
\vskip0.3cm
Let $\phi\in C^\infty(\R^3)$  supported in $(-R,R)^3$, and $\Psi\in \dH^1(\R^+\times\R^2)$ such that $\Ll\Psi\in L^2(\R^+\times\R^2)$, and $\gn(\nl\Psi)\in L^2(\R^2)$. 
From Proposition \ref{prop:apriori}, $\nl\Psi\in L^4(\R^+;L^{8/3}(\R^2))$.

In order to integrate by part rigorously, let us proceed to a regularization.

For every $\eps>0$ consider the function $\bar{\eta}_\eps$ defined on $\R^2$ by
$$
\bar{\eta}_\eps(x)=\frac{1}{\eps^2}\bar{\eta}(x/\eps),\qquad x\in \R^2,
$$
where $\bar{\eta}$ is a given smooth function compactly supported, of integral 1. Then we denote
$$
\Psi_\eps=\Psi\ast \bar{\eta}_\eps.
$$
Note that we regularize in $x$ only.  
We have
$$
\nl\Psi_\eps =(\nl\Psi)\ast\bar{\eta}_\eps,\qquad \Ll\Psi_\eps=(\Ll\Psi)\ast \bar{\eta}_\eps,\qquad \gn(\Psi_\eps)=\gn(\Psi)\ast \bar{\eta}_\eps.
$$
Hence, using Proposition \ref{prop:apriori} , for every fixed $\eps>0$, and $n\geq0$
\begin{equation}\label{eq:reg}
\nl\Psi_\eps\in C^0(\R^+,\times \dH^{2n}(\R^2))\cap L^2(\R^+,\times \dH^{2n}(\R^2)),
\end{equation}
and we have
\begin{equation}\label{eq:limit}
\begin{array}{l}
\nl\Psi_\eps\stackrel{\eps\to \infty}{\longrightarrow} \nl\Psi,\qquad \mathrm{in} \ \ L^2(\R^+\times\R^2),\\
\nl\Psi_\eps\stackrel{\eps\to \infty}{\longrightarrow} \nl\Psi,\qquad \mathrm{in} \ \  L^4(\R^+;L^{8/3}(\R^2)),\\
\Ll\Psi_\eps\stackrel{\eps\to \infty}{\longrightarrow} \Ll\Psi,\qquad \mathrm{in} \ \ L^2(\R^+\times\R^2),\\
\gamma_0(\nl\Psi_\eps)\stackrel{\eps\to \infty}{\longrightarrow} \gamma_0(\gn\Psi),\qquad \mathrm{in} \ \ L^2(\R^2).
\end{array}
\end{equation}
Since $L^{4/3}(\R^2)$ is continuously imbedded in $\dH^{-1/2}(\R^2)$, we have
$$
\bar{\div}
\left(\bar{\nabla}^\perp\Psi_\eps\otimes\nl\Psi_\eps\right)\stackrel{\eps\to \infty}{\longrightarrow}
\bar{\div}
\left(\bar{\nabla}^\perp\Psi\otimes\nl\Psi\right),\qquad \mathrm{in} \ \ L^2(\R^+;\dH^{-3/2}(\R^2)).
$$
Since $\nabla\bar{\nabla}\Psi_\eps\in L^2(\R^+\times\R^2)$ and $\bar{\nabla} \Ll \Psi_\eps\in L^2(\R^+\times\R^2)$, we have 
\begin{eqnarray*}
&&\div(\bar{\nabla}^\perp\Psi_\eps\cdot\bar{\nabla}\nl\Psi_\eps)=\bar{\nabla}^\perp\nabla\Psi_\eps:\bar{\nabla}\nl\Psi_\eps+\bar{\nabla}^\perp\Psi_\eps\cdot\bar{\nabla}(\Ll\Psi_\eps),\\
&&\qquad =\bar{\nabla}^\perp\bar{\nabla}\Psi_\eps:\bar{\nabla}\bar{\nabla}\Psi_\eps+\lambda \partial_z(\bar{\nabla}^\perp\Psi_\eps)\cdot \partial_z(\bar{\nabla}\Psi_\eps)+\bar{\nabla}^\perp\Psi_\eps\cdot\bar{\nabla}(\Ll\Psi_\eps),\\
&&\qquad =\bar{\nabla}^\perp\Psi_\eps\cdot\bar{\nabla}(\Ll\Psi_\eps).
\end{eqnarray*}
Especially, it lies in $L^2(\R^+\times\R^2)$.  
\vskip0.3cm
Using Proposition \ref{prop:Hodgeextended}, the divergence theorem,  the last computation, and the incompressibility of the flow in $\R^+\times\R^2$, we find
\begin{eqnarray*}
&&\qquad  \int_0^\infty\int_{\R^2}\nabla\phi\cdot\bar{\div}(\bar{\nabla}^\perp\Psi_\eps\otimes\nl\Psi_\eps)\,dx\,dz\\
&& =-\int_0^\infty\int_{\R^2}\phi\ \div\left(\bar{\div}(\bar{\nabla}^\perp\Psi_\eps\otimes \nl\Psi_\eps)\right)\,dx\,dz\\
&&\qquad +\int_{\R^2}\phi(0,x) \ \bar{\div}\left(\bar{\nabla}^\perp\Psi_\eps(0,x)\gn(\nl\Psi_\eps)(x)\right)\,dx\\
&&=-\int_0^\infty\int_{\R^2}\phi\ \div\left(\bar{\nabla}^\perp\Psi_\eps\cdot\bar{\nabla} \nl\Psi_\eps)\right)\,dx\,dz\\
&&\qquad -\int_{\R^2}\bar{\nabla}\phi(0,x)\cdot\bar{\nabla}^\perp\Psi_\eps(0,x)\gn(\nl\Psi_\eps)\,dx\\
&&=-\int_0^\infty\int_{\R^2}\phi\ \bar{\nabla}^\perp\Psi_\eps\cdot\bar{\nabla}(\Ll\Psi_\eps)\,dx\,dz\\
&&\qquad -\int_{\R^2}\bar{\nabla}\phi(0,x)\cdot\bar{\nabla}^\perp\Psi_\eps(0,x)\gn(\nl\Psi_\eps)\,dx\\
&&=\int_0^\infty\int_{\R^2}\left(\bar{\nabla}\phi\cdot \bar{\nabla}^\perp\Psi_\eps\right)\Ll\Psi_\eps\,dx\,dz\\
&&\qquad -\int_{\R^2}\left(\bar{\nabla}\phi(0,x)\cdot\gamma_0(\bar{\nabla}^\perp\Psi_\eps)\right)\gn(\nl\Psi_\eps)\,dx.
\end{eqnarray*}
Passing into the limit when $\eps$ goes to 0, we find that  for any  $\phi\in C^\infty(\R^3)$  supported in $(-R,R)^3$, and $\Psi\in \dH^1(\R^+\times\R^2)$ such that $\Ll\Psi\in L^2(\R^+\times\R^2)$,  and $\gn(\nl\Psi)\in L^2(\R^2)$
\begin{equation}\label{eq:1}
\begin{array}{l}
\ds{\qquad\int_0^\infty\int_{\R^2}\nabla\phi\cdot
\left[\bar{\div}(\bar{\nabla}^\perp\Psi\otimes\nl\Psi)\right]\,dx\,dz}\\[0.3cm]
\ds{=\int_0^\infty\int_{\R^2}\left(\bar{\nabla}\phi\cdot \bar{\nabla}^\perp\Psi\right)\Ll\Psi\,dx\,dz}\\[0.3cm]
\ds{\qquad -\int_{\R^2}\left(\bar{\nabla}\phi(0,x)\cdot\bar{\nabla}^\perp\gamma_0(\bar{\nabla}\Psi)\right)\gn(\nl\Psi)\,dx}.
\end{array}
\end{equation}
\vskip0.3cm
We have also
\begin{eqnarray*}
&&-\beta\int_0^\infty\int_{\R^2} \nabla\phi\P_{\lambda}(\Psi e_1)\,dx\,dz=-\beta\int_0^\infty\int_{\R^2}( \partial_1\phi)\Psi\,dx\,dz\\
&&\qquad =\beta\int_0^\infty\int_{\R^2}\phi \ \partial_1\Psi\,dx\,dz=\beta\int_0^\infty\int_{\R^2}\phi \bar{\nabla}^{\perp}\Psi\cdot\bar{\nabla}x_2\,dx\,dz.
\end{eqnarray*}
\vskip0.3cm
Hence, for every $\Phi\in C^\infty(\R^4)$ supported in $(-T,T)\times(-R,R)^3$, and every $\Psi\in C^0(0,T;\dH^1(\R^+\times\R^2))$ such that $\Ll\Psi\in L^\infty(0,T;L^2(\R^+\times \R^2))$ and $\nl\Psi\in L^\infty(0,T;L^2(\R^2))$, we have
\begin{eqnarray*}
&&\int_{0}^\infty\int_{\R^2} \nabla\Phi(0,z,x)\nl\Psi^0(z,x)\,dx\,dz\\
&&-\int_0^T\int_{0}^\infty\int_{\R^2} [\partial_t\nabla\Phi\nl\Psi+\nabla\Phi\cdot(\nl F+\beta\P_{\lambda}(\Psi e_1))]\,dx\,dz\,dt\\
&&-\int_0^T\int_{0}^\infty\int_{\R^2} \bar{\nabla}\nabla\Phi:
(\bar{\nabla}^\perp\Psi\otimes\nl\Psi)\,dx\,dz\,dt\\
&&\qquad\qquad\qquad\qquad=\\
&&-\int_{0}^\infty\int_{\R^2} \Phi(0,z,x)\Ll\Psi^0(z,x)\,dx\,dz\\
&&+\int_0^T\int_{0}^\infty\int_{\R^2} [\partial_t\Phi\Ll\Psi+\Phi \fL]\,dx\,dz\,dt\\
&&+\int_0^T\int_{0}^\infty\int_{\R^2} \bar{\nabla}\Phi\cdot\bar{\nabla}^\perp\Psi(\Ll\Psi+\beta x_2)\,dx\,dz\,dt\\
&&\qquad\qquad\qquad\qquad+\\
&&\int_{\R^2} \Phi(0,0,x)\gn(\nl\Psi^0)\,dx\\
&&-\int_0^T\int_{\R^2} [\partial_t\Phi(t,0,x)\gn(\nl\Psi)+\Phi(t,0,x) \fnu]\,dx\,dt\\
&&-\int_0^T\int_{0}^\infty\int_{\R^2} \bar{\nabla}\Phi(t,0,x)\cdot\gamma_0(\bar{\nabla}^\perp\Psi)\gn(\nl\Psi)\,dx\,dt.
\end{eqnarray*}
\vskip0.3cm
If $\Psi$ verifies the third statement of Theorem \ref{theo:equivalence}, taking $\Phi_2=\nabla\phi$ gives that it is solution to (\ref{eq:QGreformule}), so the second statement is true. If $\Psi$ is solution to (\ref{eq:QGreformule}), taking $\Phi$ compactly supported in $z$ in $\R^+$, we find that it is solution to  (\ref{eq:QG}). Taking general $\Phi$, it gives now that (\ref{eq:QGtrace}) is also verified. 
Now if   $\Psi$ is solution to  (\ref{eq:QG}) (\ref{eq:QGtrace}) (\ref{eq:QGinitial}), then it is solution to  (\ref{eq:QGreformule}).
The weak formulation of (\ref{eq:QGreformule}) is still valid for test functions of the form $\phi_1(t)\nabla\phi_2(z,x)$, $\phi_1\in C^\infty_c$ and $\phi_2\in \mathcal{S}(\R^+\times\R^2)$, and 
\begin{eqnarray*}
&&D=-\int_0^T\dt\phi_1\nl\Psi+\int\phi_1(t)\nl F-\bar{\div}\int_0^T\phi_1(t)(\bar\nabla^\perp\Psi\otimes\nl\Psi)\\
&&\qquad\qquad -\beta\int_0^T\phi_1(t)\Psi \ e_1
\end{eqnarray*}
lies in $L^2(\R^+;H^1(\R^2)+\dH^{-3/2}(\R^2))$ and in $\mathcal{S}'(\R^+\times\R^2)$,
since $\nabla^\perp\Psi\otimes\nl\Psi\in L^2(\R^+;L^{4/3}(\R^2))$ wich is imbedded in $L^2(\R^+;\dH^{-1/2}(\R^2))$. Hence, from Proposition \ref{prop:Hodgeextended}, 
$$
\P_{\lambda} D=0.
$$
 \end{proof}

\section{Proof of  the stability Theorem \ref{prop:stability} }

 We start with the following Proposition.
 \begin{prop}\label{prop:convergence}
 Take  $1<q\leq\infty$, and $T>0$. Consider $\fL^n$ and $\fnu^n$ uniformly bounded, respectively, in $ L^q(0,T;L^2(\R^+\times\R^2))\cap  L^q(0,T;L^{6/5}(\R^+\times\R^2))$ and in $ L^q(0,T;L^2(\R^2))\cap  L^{q}(0,T;L^{4/3}(\R^2))$. Consider $\eps_n\geq0$ converging to zero.
Let $\Psi_n$ be a sequence of weak solutions to 
\begin{equation*}
\left\{
\begin{array}{l}
\partial_t\nabla_\lambda\Psi_n+\P_{\lambda}(\bar\nabla\Psi_n^\perp\cdot\bar\nabla\nabla_\lambda\Psi_n)\\
\qquad\qquad=\eps_n(\bar{\Delta}^{1/2}\nl\Psi+\bar{\Delta}^{3/2}\nl\Psi)+\beta\  \P_{\lambda}(\Psi_n e_1)+\nl F_n,
\end{array}\right.\end{equation*}
such that 
$\nl \Psi_n\in C^0(0,T;L^2_{\mathrm{loc}}(\R^+\times\R^2))$, $\Ll \Psi_n$ and $\gn(\nl \Psi_n)$ are uniformly bounded, respectively, in the following spaces $L^\infty(0,T; L^2(\R^+\times\R^2))$, 
$L^\infty(0,T; L^2(\R^+\times\R^2))$, and $ L^\infty(0,T;L^2(\R^2))$. Assume that the initial values $\nl\Psi^0_n$ converges to $\nl\Psi^0$ in $L^2(\R^+\times\R^2)$. Then, there exists $\Psi$ solution 
to  (\ref{eq:QG}) (\ref{eq:QGtrace}) (\ref{eq:QGinitial}), 
such that, up to a subsequence, $\nabla \Psi_n$ converges to $\nabla \Psi$ in 
$C^0(0,T; L^2_{\mathrm{loc}}(\R^+\times \R^2))$.
 \end{prop}
Theorem \ref{prop:stability} is a direct consequence of this proposition taking $\eps_n=0$ for all $n$, and using Theorem \ref{theo:equivalence}.
\begin{proof}
From Proposition \ref{prop:apriori},  $\nl \Psi_n$ and $\nabla \Psi_n$ are uniformly bounded in 
$L^\infty(0,T;L^2(\R^+\times\R^2))$, and   $L^\infty(0,T;L^2(\R^+;H^{1/2}(\R^2)))$.
Considering the vertical component of $\nabla \Psi_n$, this gives that $\partial_z\bar{\Delta}^{1/2}\Psi_n$ is uniformly bounded in $L^\infty(0,T;L^2(\R^+;\dH^{-1/2}(\R^2)))$, and using the horizontal component, $\bar{\Delta}^{1/2}\Psi_n$ is uniformly bounded in $L^\infty(0,T;L^2(\R^+;\dH^{1/2}(\R^2)))$.  So, using  (\ref{eq:Interpolation}), this gives 
that $\bar{\Delta}^{1/2}\Psi_n$, and so $\bar{\nabla}\Psi_n$ is uniformly bounded in $L^\infty(0,T;\dH^{1/2}(\R^+\times\R^2))$. 
\vskip0.3cm
The sequence  $\partial_z(\lambda\partial_z\Psi_n)=-\bar{\Delta}\Psi_n+\Ll\Psi_n$, is uniformly bounded in $L^\infty(0,T;L^2(\R^+;H^{-1/2}(\R^2)))$. We have already shown that $\bar{\nabla}(\lambda\partial_z\Psi_n)\in L^\infty(0,T;L^2(\R^+;H^{-1/2}(\R^2)))$. Hence, the sequence $\lambda\partial_z\Psi_n$ is  bounded in $L^\infty(0,T;L^2(\R^+;H^{1/2}(\R^2)))$ uniformly with respect to $n$. So, by interpolation, $\lambda\partial_z\Psi_n$ is bounded in $L^\infty(0,T;\dH^{1/2}(\R^+\times\R^2))$.
Finally  the sequence $\nl\Psi_n$ is uniformly bounded in $L^\infty(0,T;H^{1/2}(\R^+\times\R^2))$.
\vskip0.3cm
From Theorem \ref{theo:equivalence},  $\dt\nl\Psi_n$  is  bounded in $L^\infty(0,T;L^2(\R^+;H^{-3}(\R^2)))$ uniformly with respect to $n$.
\vskip0.3cm
For every $T>0$, $R>0$, consider 
\begin{eqnarray*}
&&X_{-1}=L^2(0,R;H^{-3}((-R,R)^2)), \qquad X_0=L^2((0,R)\times(-R,R)^2),\\
&& \qquad X_1=H^{1/2}((0,R)\times(-R,R)^2).
\end{eqnarray*}
So those Banach spaces are such that $X_0$ is continuously imbedded in $X_{-1}$ and $X_1$ is compactly imbedded in $X_0$.
We have shown that $\dt\nl\Psi_n$ is uniformly bounded in $L^\infty(0,T;X_{-1})$, $\nl\Psi_n$ is uniformly bounded in $C^0(0,T;X_0)$, and in $L^\infty(0,T;X_1)$. Hence, using the Aubin-Lions lemma, if $\nl\psi^0_n$ converges to $\nl\Psi_0$ in $X_0$, then 
$\nl\Psi_n$ converges, up to a subsequence,  to $\Psi$ in $C^0(0,T,X_0)=C(0,T;L^2((-R,0)\times(-R,R)^2))$.
So,  passing into the limit of the weak formulation of (\ref{eq:QGreformule}), up to a subsequence, $\nl\Psi_n$ converges to a solution to the quasi-geotrophic equation.
\end{proof}

\section{Proof of  the existence Theorem \ref{theo:main} }

This section is dedicated to the construction of the solution to the quasi-geostrophic equation.  It is done in several steps.
\vskip0.3cm
\noindent
{\it First step:}  For every $R\in L^\infty (0,T; L^2(\R^+;\dH^{-1}(\R^2)))$, we construct a solution to 
\begin{equation}\label{eq:*}
\begin{array}{l}
\ds{\partial_t\nl\Psi+\P_{\lambda} R=\eps (\bar{\Delta}^{1/2}\nabla_\lambda\Psi+\bar{\Delta}^{3/2}\nabla_\lambda\Psi) }\\
\ds{\nabla_\lambda\Psi(t=0)=\nl\Psi^0}.
\end{array}
\end{equation}
To do this, we first notice that $\P_\lambda R\in L^\infty(0,T;L^2(\R^+;\dH^{-1}(\R^2)))$. Then we construct the solution $z$ by $z$, in Fourier space in $x$ only, to
\begin{equation*}
\begin{array}{l}
\ds{\partial_t\Phi+\P_{\lambda} R=\eps (\bar{\Delta}^{1/2}\Phi+\bar{\Delta}^{3/2}\Phi) }\\
\ds{\Phi(t=0)=\nl\Psi^0}.
\end{array}
\end{equation*}
Then we find $\nl\Psi=\P_{\lambda}\Phi$.

\vskip0.3cm
\noindent{\it Step2:} We consider again the regularization in $x$ only $\bar{\eta}_\delta$. For every $\nl \tilde\Psi\in L^\infty(0,T;L^2(\R^+\times\R^2))$ we define $T_{\delta}(\nl \tilde\Psi)=\nl\Psi$ as the solution to (\ref{eq:*}) with 
$$
R=\left(\bar{\nabla}^\perp\tilde\Psi\ast\bar\eta_\delta\right)\cdot \bar{\nabla}\nl\tilde\Psi.
$$
As long as 
$$
\|\nl\tilde\Psi(t)\|_{L^2(\R^+\times\R^2)}\leq 2 \|\nl\Psi^0\|_{L^2(\R^+\times\R^2)},
$$
we have

\begin{eqnarray*}
&&\qquad\frac{d}{dt}\int_0^\infty\int_{\R^2} |T_{\delta}(\nl\tilde\Psi)|^2\,dz\,dx\\
&&+\eps\left\{\int_0^\infty\int_{\R^2}|\bar{\Delta}^{1/4}T_{\delta}(\nl\tilde\Psi)|^2+|\bar{\Delta}^{3/4} T_{\delta}(\nl\tilde\Psi)|^2\right\}\,dz\,dx\\
&&\qquad \leq \int_0^\infty\int_{\R^2} (\bar{\nabla} T_{\delta}(\nl\tilde\Psi)) :[\bar\nabla^\perp(\tilde\Psi\ast\bar{\eta}_\delta)\otimes \nabla_\lambda \tilde\Psi]\,dz\,dx.\\
&&\qquad\leq \frac{\eps}{2}\|\bar{\Delta}^{1/2}T_{\delta}(\nl\tilde\Psi)\|^2_{L^2(\R^+\times\R^2)}\\
&&\qquad\qquad +\frac{2}{\eps} \|\nl\tilde{\Psi}\|^2_{L^2(\R^+\times\R^2)}\|\tilde{\Psi}\ast \bar{\nabla}\bar{\eta}_{\delta}\|^2_{L^\infty(\R^+\times\R^2)}.
\end{eqnarray*}
But we have 
\begin{eqnarray*}
&&\qquad\qquad\|\tilde\Psi\ast \bar\nabla\bar\eta_\delta\|_{L^\infty(\R^+\times\R^2)}\\
&&\leq C \left(\|\partial_z\tilde\Psi\ast\bar\nabla\bar\eta_\delta\|_{L^2(\R^+;L^\infty(\R^2))}+\|\tilde\Psi\ast\bar\nabla\bar\eta_\delta\|_{L^2(\R^+;L^\infty(\R^2))}\right)\\
&&\leq C \|\nl\tilde{\Psi}\|_{L^2(\R^+\times\R^2)}\left(\|\bar\nabla\bar\eta_\delta\|_{L^2}+\|\bar\eta_\delta\|_{L^2}\right)\\
&&\leq C\left(\frac{1}{\delta^2}+\frac{1}{\delta}\right)\|\nl\tilde{\Psi}\|_{L^2(\R^+\times\R^2)}.
\end{eqnarray*}
So, for $\delta<1$, we have 
\begin{eqnarray*}
&&\frac{d}{dt}\int_0^\infty\int_{\R^2} |T_{\delta}(\nl\tilde\Psi)|^2\,dz\,dx\\
&&\qquad\leq \frac{C}{\eps\delta^4}\|\nl\tilde\Psi\|^4_{L^2(\R^+\times \R^2)},
\end{eqnarray*}
and we have

$$
\|T_{\delta}(\nl\tilde\Psi)(t)\|^2_{L^2(\R^+\times\R^2)}\leq  \|\nl\Psi^0\|^2_{L^2(\R^+\times\R^2)}+\frac{2C}{\eps\delta^4}t\|\nl\Psi^0\|^4_{L^2(\R^+\times\R^2)}.
$$
So for 
$$
t_0= \frac{\eps\delta^4}{4C\|\nl\Psi^0\|^2_{L^2}},
$$
the function $T_\delta$ goes from the Banach space
$$
\mathcal{C}=\left\{f\in C^0(0,t_0;L^2(\R^+\times\R^2)); \ \|f\|_{C^0(0,t_0;L^2(\R^+\times\R^2))\leq 2\|\nl\Psi^0\|_{L^2}}\right\}
$$
into itself. Moreover, $T_\delta$ is contractive on $\mathcal{C}$ (with constant $1/2$). Hence, from the Picard fixed point Theorem, There exists a fixed point on $\mathcal{C}$. But then, the fixed point $\nl\Psi_\delta$ is solution on $[0,t_0]$ to
$$
\partial_t\nl\Psi_\delta+\P_\lambda(\left(\bar{\nabla}^\perp\Psi_\delta\ast\bar\eta_\delta\right)\cdot \bar{\nabla}\nl\Psi_\delta)=\eps (\bar{\Delta}^{1/2}\nabla_\lambda\Psi_\delta+\bar{\Delta}^{3/2}\nabla_\lambda\Psi_\delta).
$$ 
So, for every  $t\in [0,t_0]$, $\|\nl\Psi_\delta(t)\|_{L^2}\leq\|\nl\Psi^0\|_{L^2}$. And we can construct in the same way the solution recursively on $[nt_0,(n+1)t_0]$ and finally for $t\in \R$.
\vskip0.3cm
The energy inequality gives that 
\begin{equation}\label{uneetoile}\begin{array}{rcl}
&&\qquad\qquad \|\nl\Psi_\delta\|^2_{L^\infty(0,T;L^2(\R^+\times\R^2))}\\
&&+\eps\left(\|\bar{\Delta}^{1/4}\nl\Psi_\delta\|^2_{L^2((0,T)\times\R^+\times\R^2)}+\|\bar{\Delta}^{3/4}\nl\Psi_\delta\|^2_{L^2((0,T)\times\R^+\times\R^2)}\right)\\
&&\qquad\qquad \leq \|\nl\Psi^0\|^2_{L^2(\R^+\times\R^2)}.
\end{array}\end{equation}
At $z=0$, we have
\begin{eqnarray*}
&&\partial_t\gn(\nl\Psi_\delta)+\left(\bar{\nabla}^\perp\gamma_0(\Psi_\delta)\ast\bar\eta_\delta\right)\cdot \bar{\nabla}\gn(\nl\Psi_\delta))\\
&&\qquad\qquad=\eps (\bar{\Delta}^{1/2}\gn(\nabla_\lambda\Psi_\delta)+\bar{\Delta}^{3/2}\gn(\nabla_\lambda\Psi_\delta)).
\end{eqnarray*}
The energy inequality on the plane $z=0$ gives
\begin{equation}\label{deuxetoiles}\begin{array}{rcl}
&&\qquad\qquad \|\gn(\nl\Psi_\delta)\|^2_{L^\infty(0,T;L^2(\R^2))}\\
&&+\eps\left(\|\bar{\Delta}^{1/4}\gn(\nl\Psi_\delta)\|^2_{L^2((0,T)\times\R^2)}+\|\bar{\Delta}^{3/4}\gn(\nl\Psi_\delta)\|^2_{L^2((0,T)\times\R^2)}\right)\\
&&\qquad\qquad \leq \|\gn(\nl\Psi^0)\|^2_{L^2(\R^2)}.
\end{array}\end{equation}


\vskip0.3cm
\noindent{\it Step3:} Taking the divergence of the equation we find that 
\begin{eqnarray*}
&&\qquad\frac{d}{dt}\int_0^\infty\int_{\R^2} |\Ll\Psi_\delta|^2\,dz\,dx\\
&&+\eps\left\{\int_0^\infty\int_{\R^2}|\bar{\nabla}\bar{\Delta}^{-1/4}\Ll\Psi_\delta|^2+|\bar{\nabla}\bar{\Delta}^{1/4}\Ll\Psi_\delta|^2\right\}\,dz\,dx\\
&&\qquad \leq \int_0^\infty\int_{\R^2} (\Ll\Psi_\delta) \bar{\nabla}\nabla(\Psi_\delta\ast\bar{\eta}_\delta):\bar\nabla^\perp \nl\Psi_\delta\,dz\,dx.\\
&&\leq \|\bar{\nabla}\nabla\Psi_\delta\|_{L^\infty(\R^+;L^4(\R^2))}\|\bar{\nabla}\nl\Psi_\delta\|_{L^2(\R^+;L^4(\R^2))}\|\Ll\Psi_\delta\|_{L^2(\R^+\times\R^2)}\\
&&\leq \|\bar{\Delta}^{3/4}\nabla\Psi_\delta\|_{L^\infty(\R^+;L^2(\R^2))}\|\bar{\Delta}^{3/4}\nl\Psi_\delta\|_{L^2(\R^+;L^2(\R^2))}\|\Ll\Psi_\delta\|_{L^2(\R^+\times\R^2)}.
\end{eqnarray*}

From Proposition \ref{prop:apriori},
\begin{eqnarray*}
&&\qquad \|\bar{\Delta}^{3/4}\nabla\Psi_\delta\|_{L^\infty(\R^+;L^2(\R^2))}\\
&&\leq \|\bar{\Delta}^{3/4}\nabla\Psi_\delta\|_{L^2(\R^+\times\R^2)}+ \|\bar{\Delta}^{3/4}\Ll\Psi_\delta\|_{L^2(\R^+\times\R^2)}\\
&&\qquad\qquad +\|\gn(\bar{\Delta}^{3/4}\nl\Psi_\delta)\|_{L^2(\R^2)}.
\end{eqnarray*}

It gives 
\begin{eqnarray*}
&&\qquad\frac{d}{dt}\int_0^\infty\int_{\R^2} |\Ll\Psi_\delta|^2\,dz\,dx\\
&&+\eps\left\{\int_0^\infty\int_{\R^2}|\bar{\nabla}\bar{\Delta}^{-1/4}\Ll\Psi_\delta|^2+|\bar{\nabla}\bar{\Delta}^{1/4}\Ll\Psi_\delta|^2\right\}\,dz\,dx\\
&&\leq\left(\|\bar{\Delta}^{3/4}\nabla\Psi_\delta\|_{L^2(\R^+\times\R^2)}+ \|\bar{\Delta}^{3/4}\Ll\Psi_\delta\|_{L^2(\R^+\times\R^2)}+\|\gn(\bar{\Delta}^{3/4}\nl\Psi_\delta)\|_{L^2(\R^2)}\right)\\
&&\qquad\qquad  \|\bar{\Delta}^{3/4}\nabla\Psi_\delta\|_{L^2(\R^+;L^2(\R^2))}\|\Ll\Psi_\delta\|_{L^2(\R^+\times\R^2)}\\
&& \leq \|\bar{\Delta}^{3/4}\nabla\Psi_\delta\|^2_{L^2(\R^+\times\R^2)}+\frac{\eps}{2} \|\bar{\Delta}^{3/4}\Ll\Psi_\delta\|^2_{L^2(\R^+\times\R^2)}+\|\gn(\bar{\Delta}^{3/4}\nl\Psi_\delta)\|^2_{L^2(\R^2)}\\
&&\qquad\qquad +(2+\frac{2}{\eps})\|\bar{\Delta}^{3/4}\nabla\Psi_\delta\|^2_{L^2(\R^+\times\R^2)}\|\Ll\Psi_\delta\|^2_{L^2(\R^+\times\R^2)}
\end{eqnarray*}

Hence:
\begin{eqnarray*}
&&\qquad\frac{d}{dt}\int_0^\infty\int_{\R^2} |\Ll\Psi_\delta|^2\,dz\,dx\\
&&\leq G_\eps(t)\left(\int_0^\infty\int_{\R^2} |\Ll\Psi_\delta|^2\,dz\,dx+1\right),
\end{eqnarray*}
with 
$$
G_\eps(t)=(\frac{2}{\eps}+2)\|\bar{\Delta}^{3/4}\nabla\Psi_\delta\|^2_{L^2(\R^+\times\R^2)}+\|\gn(\bar{\Delta}^{3/4}\nl\Psi_\delta)\|^2_{L^2(\R^2)}.
$$

Thanks to (\ref{uneetoile}) and (\ref{deuxetoiles}), the Function $G_\eps$ is bounded in $L^1(0,T)$, uniformly with respect to $\delta$.
So $\| \Ll\Psi_\delta\|_{L^\infty(0,T;L^2(\R^+\times\R^2))}$ is uniformly bounded with respect to $\delta$ and  we can pass to the limit 
$\delta\to 0$ using Proposition \ref{prop:convergence}. 
\vskip0.3cm
For $\delta=0$, we have now 
$$
\frac{d}{dt}\int_0^\infty\int_{\R^2} |\Ll\Psi_\eps|^2\,dz\,dx=0.
$$
So  $\|\Ll\Psi_\eps\|_{L^\infty(0,T;L^2(\R^+\times\R^2))}$ is uniformly bounded with respect to $\eps$, and we can pass into the limit 
when $\eps$ goes to 0 thanks to Proposition \ref{prop:convergence}. 

  \bibliography{biblio}
\bibliographystyle{plain}

\end{document}